\documentclass[12pt,reqno]{amsart}

\usepackage{amssymb}

\newtheorem{theorem}{Theorem}[section]
\newtheorem{proposition}[theorem]{Proposition}

\theoremstyle{definition}
\newtheorem{definition}[theorem]{Definition}
\newtheorem{remark}[theorem]{Remark}

\numberwithin{equation}{section}

\begin{document}

\baselineskip=15pt

\title[Rational parallelism on complex manifolds]{Rational parallelism on complex manifolds}

\author[I. Biswas]{Indranil Biswas}

\address{School of Mathematics, Tata Institute of Fundamental
Research, Homi Bhabha Road, Mumbai 400005, India}

\email{indranil@math.tifr.res.in}

\author[S. Dumitrescu]{Sorin Dumitrescu}

\address{Universit\'e C\^ote d'Azur, CNRS, LJAD, France}

\email{dumitres@unice.fr}

\subjclass[2010]{32Q57, 53C30, 53A55}

\keywords{Parallelizable manifold, Fujiki class $\mathcal C$, rational parallelism, Lie group}

\date{}

\begin{abstract}
A Theorem of Wang in \cite{Wa}  implies that any holomorphic parallelism on a 
compact complex manifold $M$ is flat with respect to some complex Lie algebra structure
whose dimension coincides with that of $M$. 
We study here rational parallelisms on complex manifolds. We exhibit rational 
parallelisms on compact complex manifolds which are not flat with respect to any complex Lie 
algebra structure.
\end{abstract}

\maketitle

\tableofcontents

\section{Introduction}

A well-known result of Wang \cite{Wa} classifies compact complex manifolds $M$ with 
holomorphically trivial holomorphic tangent bundle $TM$.

Recall that a holomorphic trivialization of $TM$ is defined by a holomorphic one-form $\omega 
\,\in\, \Omega^1 (M,\, V)$ with values in a complex vector space $V$ with $\dim 
V\,=\,\dim_{\mathbb C} M$ such that $\omega_m \,:\, T_mM \,\longrightarrow\, V$ is a linear 
isomorphism at every $m \,\in\, M$. Hence $\omega$-constant vector fields define a parallelism 
on $M$ and trivialize $TM$. A compact complex manifold with trivial tangent bundle is called 
(holomorphically) {\it parallelizable.}

In this context Wang proved that there exists a complex Lie algebra structure $\mathfrak L$ on 
$V$ such that $\omega$ realizes a Lie algebra isomorphism between $TM$ and $\mathfrak L$ (see 
Theorem \ref{Wang} here). Notice that this is equivalent to
the assertion that $\omega$, seen as Cartan 
geometry modeled on $\mathfrak L$, is {\it flat} (\cite{Sh}, Chapter 3 and Chapter 5), meaning 
that $\omega$ satisfies the Maurer-Cartan equation with respect to the Lie algebra structure 
$\mathfrak L$. In view of Darboux-Cartan Theorem, \cite[Chapter 3, p. 116]{Sh}, this implies that $M$ 
inherits a $(L,\,L)$-structure in the sense of Ehresmann--Thurston \cite{Eh} with $L$ being a complex 
connected Lie group with Lie algebra $\mathfrak L$. The compactness assumption
on $M$ ensures that the $(L,\,L)$-structure is complete and $M$ is biholomorphic to a
quotient of $L$ by a lattice in it. This result was extended by Winkelmann to certain open complex manifolds \cite{Wi1, 
Wi2}.

We study here rational parallelisms on compact complex manifolds given by holomorphic one-forms 
$\omega\,\in\, \Omega^1 (M,\, V)$ such that $\omega_m \,:\, T_mM \,\longrightarrow\,
V$ is a linear isomorphism for $m$ 
lying in an open dense subset $U\,=\,M \setminus S$ (the subset $S$ is an analytic divisor, see Proposition \ref{divisor}). In
this case, $\omega$--constant vector fields are
actually meromorphic on $M$ with poles lying in $S$ and they 
holomorphically trivialize $TU$. This definition corresponds to the particular case of a {\it 
branched Cartan geometry}, whose model is a Lie group, as introduced and studied by the authors in 
\cite{BD}. We prove here (see Theorem \ref{flat}) that on compact complex manifolds $M$ which 
are either in the Fujiki class $\mathcal C$ (i.e. $M$ is the meromorphic image of a K\"ahler 
manifold \cite{Fu}), or of complex dimension two or of algebraic dimension zero (i.e. all 
meromorphic functions on $M$ are constant), all such rational parallelisms are flat with 
respect to some complex Lie algebra structure on $V$ (meaning $\omega$ is a Lie algebra 
isomorphism). For such a situation we deduce that the fundamental group of $M$ is infinite.
 
Flat rational parallelisms were studied in particular in \cite{BC}.
 
The main result proved here (Theorem \ref{non flat}) exhibits examples of rational 
parallelisms on compact complex manifolds which are non-flat with respect to any complex Lie 
algebra structure.

\section{Holomorphic parallelisms}

In this section we recall the proof of Wang's classification Theorem of (holomorphically) 
parallelizable manifolds.  The idea of this proof will be useful later on  in the study of flatness for  rational parallelisms.

Let us first recall the following:

\begin{definition} 
A holomorphic trivialization (parallelization) of the holomorphic tangent bundle $TM$ of a 
compact complex manifold $M$ of (complex) dimension $m$ is a holomorphic one-form $\omega\,\in\, 
\Omega^1(M, \,V)$ with values in a complex vector space $V$ of dimension $m$ such that $\omega_u 
\,:\, T_uM \,\longrightarrow\, V$ is a linear isomorphism at
every $u \,\in\, M$.
\end{definition}

\begin{theorem}[{Wang, \cite{Wa}}]\label{Wang}
Let $M$ be a compact
complex manifold of dimension $m$ and $$\omega \,:\, TM \,\longrightarrow\, V$$ a
holomorphic trivialization of its holomorphic tangent bundle. Then the universal cover of $M$
is biholomorphic to a complex Lie group $L$, and the pull-back of $\omega$ on $L$ coincides with
the Maurer-Cartan form of $L$. Consequently,
$d \omega + \frac{1}{2} \lbrack \omega,\, \omega \rbrack_{\mathfrak L} \,=\,0$,
where $\mathfrak L$ is the Lie algebra of $L$. The manifold $M$ is  biholomorphic to a quotient of $L$ by a lattice in it.
 
Moreover, $M$ is K\"ahler if and only if $L$ is abelian. In this case $M$ is a compact complex torus.
\end{theorem}

\begin{proof}
Consider a basis $(e_1,\, \cdots,\, e_m)$ of the complex vector space $V$, and let
$X_{1},\, X_{2},\, \cdots, \,X_{m}$ be global holomorphic $\omega$-constant vector fields
on $M$ such that $\omega (X_i)\,=\,e_i$ for all $i$. Consequently, the vector fields
$X_{1},\, \cdots, \,X_{m}$ span $TM$. 

Notice that for all $1\,\leq\, i,\,j \,\leq\, m$, we have $$\lbrack X_{i}, \,X_{j} \rbrack 
\,=\,f_{1}^{ij} X_{1} + f_{2}^{ij} X_{2} +\ldots+ f_{m}^{ij}X_{m}$$ with $f_{k}^{ij}$ being 
holomorphic functions on $M$. Since $M$ is compact, these (holomorphic) functions are constant 
and, consequently, $X_{1},\, X_{2},\, \cdots, \,X_{m}$ generate a $m$-dimensional complex Lie 
algebra $\mathfrak L$. When $V$ is endowed with the Lie algebra structure of $\mathfrak L$, the 
form $\omega$ produces a Lie algebra isomorphism.

By Lie's theorem, there exists a unique connected simply connected complex Lie group $L$ 
corresponding to $\mathfrak L$. The holomorphic parallelization of $M$ by $\omega$-constant 
holomorphic vector fields is locally isomorphic to the parallelization given by the left
translation-invariant vector fields on the Lie group $L$.

Since $M$ is compact, all vector fields $X_{i}$ are complete and they define a holomorphic 
locally free transitive action of $L$ on $M$ (with discrete kernel). Hence $M$ is biholomorphic 
to a quotient of $L$ by a cocompact discrete subgroup $\Gamma$ in $L$.

The Lie-Cartan formula
$$d\omega (X_{i},\,X_{j})\,=\,X_i \cdot \omega (X_j)-X_j \cdot \omega(X_i) -
\omega (\lbrack X_{i},\,X_{j} \rbrack)\,=\,-\omega (\lbrack X_{i},\,X_{j} \rbrack)\,=\,
-\lbrack \omega(X_i), \,\omega (X_j) \rbrack $$
shows that $\omega$ satisfies the Maurer-Cartan equation of the Lie group $L$, which can also be
expressed more formally as $d \omega + \frac{1}{2} \lbrack \omega, \omega \rbrack_{\mathfrak L}
\,=\,0$.

Assume now that $M$ is K\"ahler. Then, any holomorphic form on $M$ is closed. The Maurer-Cartan 
formula shows that the one-forms composing the isomorphism $\omega$ are all closed if and only 
if $L$ is abelian and thus $M$ is a compact complex torus, which is the
quotient of a complex vector space by a lattice.
\end{proof}

With the terminology of Cartan geometries \cite{Sh}, $\omega$ defines a flat Cartan geometry 
with respect to the Lie algebra structure $\mathfrak L$ on $V$. Indeed, the vanishing of 
Cartan's curvature is equivalent to the fact that $\omega$ satisfies the Maurer-Cartan 
equation $d \omega + \frac{1}{2} \lbrack \omega, \omega \rbrack_{\mathfrak L}=0$.

\section{Rational parallelisms}

We study here rational parallelisms defined by a {\it branched Cartan geometry} modeled on a 
Lie group in the sense of \cite{BD}. Let us begin with the following:

\begin{definition}
A branched holomorphic co-parallelism on the holomorphic tangent bundle $TM$ of a compact 
complex manifold $M$ of (complex) dimension $m$ is a holomorphic one-form $\omega \,\in\, 
\Omega^1(M,\, V)$ with values in a complex vector space $V$ of dimension $m$ such that 
$\omega_u \,:\, T_uM \to V$ is a linear isomorphism for $u$ in    an open dense subset $U$ in $M$ (which is 
necessarily the complement of a divisor in $M$, see Proposition \ref{divisor}).
\end{definition}

The branched holomorphic co-parallelism is {\it flat} in the sense of branched Cartan 
geometries \cite{BD}, with respect to some Lie algebra structure $\mathfrak L$ on $V$, if and 
only if $\omega$ satisfies the Maurer-Cartan equation $d \omega + \frac{1}{2} \lbrack \omega, \, 
\omega \rbrack_{\mathfrak L}\,=\,0$.

Notice that a basis $(e_1,\,\cdots,\, e_m)$ of $V$ uniquely defines a family $(X_1,\, \cdots,\, 
X_m)$ of meromorphic vector fields on $M$ such that $\omega(X_i)\,=\,e_i$. This family of 
meromorphic vector fields $X_i$ holomorphically span $TM$ at the generic point in $M$: they 
form a {\it rational parallelism} of $TM$.

Moreover, the branched holomorphic co-parallelism is flat with respect to the Lie algebra 
structure $\mathfrak L$, if and only if $\omega(\lbrack X_i, \,X_j \rbrack)  \,= \,\lbrack e_i,\, e_j 
\rbrack_{\mathfrak L}$, for all $i,\,j$.

 \begin{remark} \label{1} The pull-back of a (flat) holomorphic parallelism through a holomorphic map whose differential is invertible at the generic point gives rise to a (flat) branched holomorphic co-parallelism. Consequently,
a blow-up (or a ramified cover) of a parallelizable manifold is endowed with a flat branched holomorphic co-parallelism.
\end{remark}

The subset $M \setminus U$ where $\omega$ fails to be an isomorphism is called the {\it branching locus} of the co-parallelism.

Let us first prove:

\begin{proposition}\label{divisor}
The branching locus of a branched co-parallelism $\omega$ is either empty, or it is
an effective divisor in $M$ representing the canonical class. Consequently, if the
canonical class of $M$ is trivial (or more generally, does not have an effective
representative), any branched holomorphic co-parallelism on $TM$ has empty
branching locus and therefore $M$ is a parallelizable manifold.
\end{proposition} 

\begin{proof} Choose a basis
$(e_1,\, \cdots,\, e_m)$ of $V$ over $\mathbb C$ and consider the corresponding components 
$\omega_i$ of $\omega$ in this basis: $$\omega\,=\,(\omega_1, \,\cdots,\, \omega_m)
\,\in\, \Omega^1(M, \,\mathbb C^m)\, .$$

The branching locus of the co-parallelism $\omega$ coincides with the vanishing set of the 
$m$-form $\omega_1 \wedge \ldots \wedge \omega_m$, considered as a holomorphic section of the 
canonical bundle $K_M$. Consequently, the branching locus of $\omega$ is either empty or it 
coincides with an effective divisor representing the canonical class of $M$.

If the canonical class of $M$ does not admit an effective representative, then the 
co-parallelism $\omega$ has empty branching locus. It follows that $M$ is parallelizable 
manifold, and a quotient of a complex Lie group $L$ by a lattice in it (see Theorem \ref{Wang}).
\end{proof} 

Let us now  prove:

\begin{theorem} \label{flat}
Let $M$ be a compact complex manifold endowed with a branched holomorphic co-parallelism 
$\omega \,\in\, \Omega^1(M,\,V)$.

(i) If $M$ is either in the Fujiki class $\mathcal C$ or a complex surface, then $\omega$ is
flat with respect to the abelian Lie algebra structure on $V$;

(ii) If $M$ is of algebraic dimension zero (so all meromorphic functions on $M$ are 
constants),  then $\omega$ is flat with respect to some complex Lie algebra structure $\mathfrak 
L$ on $V$;

(iii) If $\omega$ is flat with respect to some complex Lie algebra structure $\mathfrak L$ on 
$V$, then the fundamental group of $M$ is infinite. \end{theorem}

\begin{proof} To prove (i), let us first deal with the case where $M$ is in Fujiki class $\mathcal C$ 
(meaning that $M$ is the meromorphic image of a K\"ahler manifold \cite{Fu}). Fix a basis 
$(e_1,\, \cdots,\, e_m)$ of $V$ over $\mathbb C$ and consider the corresponding components 
$\omega_i$ of $\omega$ in this basis: $$\omega\,=\,(\omega_1, \,\cdots,\, \omega_m)
\,\in\, \Omega^1(M, \,\mathbb C^m)\, .$$

The manifold $M$ being in class $\mathcal C$, by a result of Varouchas, \cite{Va},
it must be bimeromorphic to a 
K\"ahler manifold. Consequently, as for K\"ahler manifolds, all holomorphic one-forms on $M$ 
must be closed. This implies $d \omega_i \,=\,0$ and consequently $\omega$ is flat with respect 
to the abelian Lie algebra $\mathbb C^m$.

Moreover, since $M$ admits non-trivial closed holomorphic one-forms $\omega_i$, the 
abelianization of the fundamental group of $M$ is infinite. This gives the proof of statement (iii) 
when $\mathfrak L$ is abelian.

The same proof works for any compact complex surface, since any holomorphic one-form on a 
compact complex surface is closed (see, for example, \cite{Br}, p.  644).

Proof of (ii): Let us assume that $M$ is of algebraic dimension zero and is endowed with a branched 
holomorphic co-parallelism $\omega \,\in\, \Omega^1(M,\,V)$. Fix a basis $(e_1,\, \cdots, \, 
e_m)$ of $V$ over $\mathbb C$ and consider the corresponding meromorphic vector fields $X_i$ on 
$M$ defined by $\omega(X_i)\,=\,e_i$. They form a rational parallelism on $M$.

Notice that for all $1\,\leq\, i,\,j \,\leq\, m$, we have
$$\lbrack X_{i}, \,X_{j} \rbrack \,=\,f_{1}^{ij} X_{1} + f_{2}^{ij} X_{2} +\ldots+ f_{m}^{ij}X_{m}$$ 
with $f_{k}^{ij}$ being meromorphic functions on $M$. Since $M$ is of algebraic dimension zero, 
the functions $f_{k}^{ij}$ are all constants and consequently $X_{1},\, X_{2},\, \cdots, \,X_{m}$ 
generate a $m$-dimensional complex Lie algebra $\mathfrak L$. When $V$ is endowed with the Lie 
algebra structure of $\mathfrak L$, the form $\omega$ produces a Lie algebra isomorphism. Hence 
the rational parallelism is flat with respect to the structure of the Lie algebra $\mathfrak L$.

Proof of (iii): To prove by contradiction, assume that the fundamental group of $M$ is finite. 
Then, up to replacing $M$ by its universal cover endowed with the pullback of $\omega$, we shall 
assume that $M$ is simply connected.

Recall that the rational parallelism $\omega$  is supposed to be  flat with respect to the structure of the Lie algebra $\mathfrak L$.
The Darboux-Cartan Theorem (\cite{Sh}, Chapter 3, p. 116) implies that the open dense subset 
$U\, \subset\, M$ where $\omega$ is an isomorphism, inherits a $(L,\,L)$-structure in the sense 
of Ehresmann-Thurston \cite{Eh}, where $L$ is a complex connected Lie group
with Lie algebra $\mathfrak L$.

   Since $M$ is simply connected we get a holomorphic 
developing map $$d\,:\, M\,\longrightarrow\, L$$ which is a submersion on an open dense subset 
(outside the branching locus described in Proposition \ref{divisor})  (see Section 2.4 in \cite{BD}). The manifold $M$ being compact, the image 
$d(M)$ must be closed (and open) in $L$, hence $L\,=\,d(M)$ is compact. But compact complex Lie 
groups are abelian. This implies $\mathfrak L$ is abelian, and we conclude as before in (i) that 
the components $\omega_i$ of $\omega$ are closed holomorphic one-forms. The complex manifold $M$ 
being simply connected,  for each $i$ there exists a holomorphic function $h_i$ on $M$ such that 
$\omega_i\,=\,dh_i$. But holomorphic functions on compact manifolds are constant 
and, consequently, $\omega_i\,=\,dh_i\,=\,0$, for all $i$: a contradiction.
\end{proof}

In contrast to the unbranched case (Theorem \ref{Wang}) we exhibit the following non-flat 
examples, which are inspired by a construction of non-closed holomorphic one-forms in 
\cite{Br} (p. 648).

\begin{theorem}\label{non flat}
There exists a branched holomorphic co-parallelism $\omega \,\in\,\Omega^1(P_E,\,V)$
on some compact (non-K\"ahler) principal elliptic bundle $P_E$, over the product of
two Riemann surfaces $S_1$ and $S_2$ of genus $g\,\geq\, 2$, such that $\omega$ is non-flat
with respect to any complex Lie algebra structure $\mathcal L$ on $V$.
\end{theorem}

\begin{proof}
We shall first construct a holomorphic two-form $\Omega$ on the product $S_1\times S_2$ of 
two-Riemann surfaces such that the periods of $\Omega$ belong to a lattice $\Lambda$ in 
$\mathbb C$.

Consider the standard elliptic curve $E\,=\,{\mathbb C}/\Lambda$, with $\Lambda
\,=\, \mathbb Z \oplus \sqrt{-1}\mathbb Z$. Let $z$ be the coordinate on $\mathbb C$ and $dz$
the associated standard (translation invariant) one-form
on $E$. Notice that the periods of $dz$ form the lattice $\Lambda\,=\,{\mathbb Z}\oplus
\sqrt{-1}\mathbb Z$. 

Choose two Riemann surfaces $S_1$ and $S_2$ (of genus $g \,\geq\, 2$) admitting holomorphic 
ramified covers $f_1 \,:\, S_1 \,\longrightarrow\, E$ and $f_2 \,:\, S_2 \,
\longrightarrow\, E$. Let us denote by $\Omega$ the 
holomorphic two-form on the complex surface $S_1 \times S_2$ defined by the pull-back: 
$$\pi_1^*(f_1^*dz) \wedge \pi_2^*(f_2^*dz)\, ,$$ where $\pi_1$ and $\pi_2$ are the projections of 
$S_1 \times S_2$ on the first and the second factor respectively. The periods of $\Omega$ 
belong to the lattice $\Lambda$ in $\mathbb C$.

Fix an open cover $\{U_i\}$ of $S_1 \times S_2$ such that all $U_i$ and all connected components of $U_i 
\cap U_j$ are contractible. Then on each $U_i$ there exists a holomorphic one-form $\omega_i$ 
such that $\Omega_i\,=\,d \omega_i$, where $\Omega_i$ denotes the restriction of $\Omega$ to 
$U_i$. On any intersections $U_i \cap U_j$, we have $\omega_i-\omega_j\,=\,dF_{ij}$, where
$F_{ij}$ is a holomorphic function defined on $U_i\cap U_j$.

On triple intersections $U_i \cap U_j \cap U_k$, the functions $F_{ij} +F_{jk}+F_{ki}$ form a 
locally constant two-cocycle which represents the class of $\Omega$ in $H^2(S_1 \times S_2,\, { 
\mathbb C})$. Hence we can choose the forms $\omega_{i}$ and the associated functions $F_{ij}$ 
in such a way that $F_{ij}+F_{jk}+F_{ki }$ belongs to the lattice
$\Lambda$ of periods for every triple $i,\, j,\, k$.

Consider then every holomorphic function $F_{ij}$ as taking values in the translations group $E = 
{\mathbb C}/ \Lambda$. In this way we construct an associate one-cocycle with values in $E$. Let 
us form the corresponding holomorphic principal elliptic bundle $$\pi : P_E\, \longrightarrow\, 
S_1 \times S_2$$ with $E$ as the structure group.

On each local trivialization $U_i \times E$ of $P_E$, consider the local one-form 
$p_1^*(\omega_i)+p_2^*dz $, where $p_1$ and $p_2$ are the projections of $U_i\times E$ on the 
first and the second factor respectively. By construction, these local one-forms glue to a 
global holomorphic one-form $\theta$ on the principal elliptic bundle $P_E$. Moreover we have
$d\theta\,=\,\pi^*\Omega$. Since $\Omega$ does not vanish, $\theta$ is a non-closed 
holomorphic one-form on $P_E$, in particular the complex manifold $P_E$ is non-K\"ahler. In 
fact $\theta$ is a holomorphic connection form on the principal elliptic bundle $P_E$ whose
curvature is $\pi^*\Omega$ \cite{At}.

Consider a section $\omega_1$ of the canonical bundle $K_{S_1}$ of $S_1$ such that $\omega_1\,= 
\,g_1 \cdot f_1^*(dz)$, where $g_1$ is a non-constant meromorphic function on $S_1$. Also
consider a holomorphic section $\omega_2$ of the canonical bundle $K_{S_2}$ of $S_2$ such that 
$\omega_2\,=\,g_2\cdot f_2^*(dz)$ with $g_2$ being a non-constant meromorphic function on $S_2$.

Let $$\omega\,=\,(\theta, \,\pi^* \pi_1^*\omega_1, \,\pi^*\pi_2^*\omega_2) \,\in\, 
\Omega^1(P_E,\, \mathbb C^3)$$ be the associated branched holomorphic co-parallelism on $P_E$. 
Then $$d \theta \,=\, \pi^*\Omega\,=\, (g \circ \pi) \cdot \pi^* (\pi_1^*(\omega_1) \wedge 
\pi_2^*(\omega_2))\, ,$$ with $g= (g_1^{-1} \circ \pi_1) \cdot (g_2^{-1} \circ \pi_2)$ being a 
non-constant meromorphic function on $S_1 \times S_2$. Since $g \circ \pi$ is a non-constant 
(meromorphic) function on $P_E$, this implies that the branched holomorphic parallelism $\omega$ is 
non-flat for any Lie algebra structure on the vector space ${\mathbb C}^3$. Moreover, $\omega$ 
is not locally homogeneous on any nonempty open subset in $P_E$.
\end{proof}

We have seen in Theorem \ref{flat} (iii) that compact complex manifolds admitting flat branched 
holomorphic co-parallelisms have infinite fundamental group. We do not know if there exists 
compact complex simply connected manifolds bearing branched holomorphic co-parallelism 
(necessarily non-flat with respect to any Lie algebra structure, as that in Theorem \ref{non 
flat}).

\begin{remark}
The manifold $P_E$ in Theorem \ref{non flat} is a ramified cover of the parallelizable 
manifold $H/\Gamma$, where $H$ is the complex Heisenberg group of upper triangular unipotent 
$(3 \times 3)$ matrices with complex entries and $\Gamma$ is the lattice of matrices with 
Gaussian integers as entries. This quotient $H/\Gamma$ is biholomorphic to a principal 
elliptic bundle with fiber $E\,=\,{\mathbb C}/\Lambda$ (recall that $\Lambda\,=\,{\mathbb 
Z}\oplus \sqrt{-1}{\mathbb Z}$) over the two-dimensional compact complex torus $E \times E$. 
The map
$$
f\,=\,(f_1\circ\pi_1,\, f_2\circ\pi_2) \,:\, S_1\times S_2 \,\longrightarrow\, E\times E
$$
in the proof 
of Theorem \ref{non flat} is a ramified cover. The bundle $P_E$ is the pull-back of the 
elliptic bundle $H/\Gamma \,\longrightarrow\, E \times E$ through $f$. Consequently, $P_E$ is 
a ramified cover  of $H/\Gamma$ and inherits a branched holomorphic co-parallelism which is 
flat with respect to the Lie algebra of $H$ (the three dimensional complex Heisenberg algebra) (see Remark \ref{1}).

We do not know examples of compact complex manifolds admitting branched holomorphic 
co-parallelisms and which are not ramified covers of parallelizable manifolds. In particular, 
we do not know if all compact complex manifold admitting branched holomorphic co-parallelisms 
also admit flat branched holomorphic co-parallelisms (with respect to some Lie algebra 
structure).
\end{remark}

\section*{Acknowledgements}

The second-named author 
wishes to thank T.I.F.R. Mumbai for hospitality. Both authors wishes to thank ICTS Bangalore 
and IISc Bangalore for hospitality.
 
This work has been supported by the French government through the UCAJEDI Investments in the 
Future project managed by the National Research Agency (ANR) with the reference number 
ANR2152IDEX201. The first author is partially supported by a J. C. Bose Fellowship.



\begin{thebibliography}{ZZZZ}

\bibitem[At]{At} M. F. Atiyah, Complex analytic connections in fibre
bundles, \textit{Trans. Amer. Math. Soc.} \textbf{85} (1957), 181--207.

\bibitem[BD]{BD} I. Biswas and S. Dumitrescu, Branched holomorphic Cartan geometries and 
Calabi-Yau manifolds, \textit{Int. Math. Res. Not.}  \textbf{23} (2019), 7428--7458.

\bibitem[BC]{BC} D. Bl\'azquez-Sanz and G. Casale, Parallelisms and Lie connections, 
S.I.G.M.A., \textbf{13 (86)}, (2017).

\bibitem[Br]{Br} M. Brunella, On holomorphic forms on compact complex threefolds,
\textit{Comment. Math. Helv.} \textbf{74} (1999), 642--656.

\bibitem[Eh]{Eh} C. Ehresmann, Sur les espaces localement homog\`enes, 
\textit{L'Enseign. Math.} \textbf{35} (1936), 317--333.

\bibitem[Fu]{Fu} A. Fujiki, On the structure of compact manifolds in $\mathcal C$, 
\textit{Advances Studies in Pure Mathematics}, \textbf{1}, Algebraic Varieties and 
Analytic Varieties, (1983), 231--302.

\bibitem[Sh]{Sh} R. W. Sharpe, {\it Differential Geometry :
Cartan's Generalization of Klein's Erlangen Program}, Graduate Text Math., 166,
Springer-Verlag, New York, Berlin, Heidelberg, 1997.

\bibitem[Va]{Va} J. Varouchas, K{\"a}hler spaces and proper open morphisms, 
\textit{Math. Ann.} \textbf{283} (1989), 13--52.

\bibitem[Wa]{Wa} H.-C. Wang, Complex Parallisable manifolds, \textit{Proc. Amer. 
Math. Soc.} \textbf{5} (1954), 771--776.

\bibitem[Wi1]{Wi1} J. Winkelmann, On manifolds with trivial logarithmic tangent bundle, 
\textit{ Osaka Jour. Math.} \textbf{41} (2004), 473--484.

\bibitem[Wi2]{Wi2} J. Winkelmann, On manifolds with trivial logarithmic tangent bundle: the 
non-K\"ahler case, \textit{Trans. Groups} \textbf{13} (2008), 195--209.

\end{thebibliography}
\end{document}